\documentclass[12pt]{article}
\usepackage{amssymb,amsmath,amsthm}
\newtheorem{theorem}{Theorem}
\newtheorem{corollary}{Corollary}
\newtheorem{lemma}{Lemma}
\newtheorem{proposition}{Proposition}
\theoremstyle{definition}

\theoremstyle{remark}

\begin{document}
\title{The Integral Representation of Solutions\\ to a Boundary Value Problem
on the Half-Line\\ for a System of Linear ODEs\\ with Singularity of
First Kind} \author{Yulia Horishna\thanks{National Taras Shevchenko
University of Kyiv, Volodymyrs'ka 64, Kyiv, 01033, Ukraine},\and
Igor Parasyuk\thanks{National Taras Shevchenko University of Kyiv,
Volodymyrs'ka 64, Kyiv, 01033, Ukraine}, \and  Lyudmyla
Protsak\thanks{National Pedagogical Dragomanov University, Pirogova
9, Kyiv, 01601, Ukraine}}
 \maketitle
\abstract{We consider a problem of finding vanishing at infinity
$C^1([0,\infty))$-solutions
 to non-homogeneous system of linear ODEs which
has the pole of first order at $x=0$. The resonant case where the
corresponding homogeneous problem has nontrivial solutions is of main
interest. Under the conditions that the homogeneous system is
exponentially dichotomic on $[1,\infty)$ and the residue of system's
operator at $x=0$ does not have eigenvalues with real part 1, we
construct the so-called generalized Green function. We also establish
conditions under which the main non-homogeneous problem can be
reduced to the Noetherian one with nonzero index.}

\section{Introduction}
 In the space $\mathbb R^n$ endowed with a scalar product
 $\langle \cdot,\cdot \rangle $ and the corres\-pond\-ing norm $\|\cdot \|$
 the following linear singular system is
 considered:
\begin{equation}\label{eq:l-s-sys}
  y'= \left(\frac{A}{x}+B(x)\right)y+\frac{a}{x}+f(x).
\end{equation}
Here $A\in \mathrm{Hom}(\mathbb R^n)$ is a linear operator, $a\in \mathbb R^n$ is a
constant vector, $B(\cdot):[0,\infty) \mapsto \mathrm{Hom}( \mathbb R^n)$ and $f(\cdot):[0,\infty)
\mapsto \mathbb R^n$ are continuous bounded mappings for which there
exists a constant $M>0$ such that $\|B(x)\|\leq M$ and $\|f(x)\|\leq
M$ for all $x\in [0,\infty)$. (The norm of a linear operator in $\mathbb
R^n$ is considered to be concordant with the norm in  $\mathbb R^n$.)

We seek a  solution $y(x)$ of the system (\ref{eq:l-s-sys}) which
satisfies the follow\-ing conditions:
\begin{gather}\label{eq:l-s-bvp}
  y(\cdot)\in C^1\big([0,\infty) \mapsto \mathbb R^n\big),\quad  y(+\infty)=0.
\end{gather}

The stated problem belongs to the class of singular ones on account
of both having a singularity at the point $x=0$ and unboundedness of
the interval where the independent variable is defined. The problems
of such a kind often arise when constructing and investigating
solutions of various equations of mathematical phy\-sics. Majority of
papers devoted to study of such problems deal with second and higher
order equations (see e.g. \cite{Kig75}--\cite{DjeMou06}). Despite the
fact that corresponding bibliography amounts to several hundreds of
titles, we failed to find a ready-made procedure for establishing
existence conditions and integ\-ral representation of solutions to
the problem (\ref{eq:l-s-sys})--(\ref{eq:l-s-bvp}). The necessity of
such representation naturally arises when solving the problem about
perturbations of solutions to singular non-linear boundary value
prob\-lems on the semi-axis \cite{PPP03, PPP06}.

While considering the above problem, we did not exclude the
so-called resonance case when the corresponding homogeneous problem
has non-trivial solutions. In this connection results of papers
\cite{Har70}--\cite{Sam02} should be mentioned, which are devoted to
the problem of existence of solutions to linear non-homogeneous
systems bounded on the entire axis, in particular, extension of
Fredholm and Noether theory over such systems. It should be noted
that in papers \cite{Kig03}--\cite{AgaKig06} the authors find quite
general sufficient conditions for boundary value problems on a
finite interval with non-integrable singularities to have the
Fredholm property with index zero.

The  present paper is organized as follows. Section
~\ref{sec:FundOper} contains an auxiliary result about the structure
of a fundamental operator of a linear homogeneous system with
continuous (however non-analytic) coefficients on the interval
$(0,x_0)$ and singular point of the first kind at $x=0$. In section
~\ref{sec:AddCond}, we describe additional conditions imposed on the
linear homogeneous system and classify its solutions in accordance
with their asymptotical behavior when $x\to +0$ and $x\to +\infty $. In
section ~\ref{sec:GenGreenFunc}, the existence criterion for the
solution to a boundary value problem with homogeneous boundary
conditions is established and the Green function for this problem is
constructed. Finally, in section ~\ref{sec:MainRes}, the main result
is stated --- the theorem about existence and integral representation
of solutions to the problem (\ref{eq:l-s-sys})--(\ref{eq:l-s-bvp}).

\section{The structure of the fundamental operator of
linear system near the singular point of  first
kind}\label{sec:FundOper} Consider the linear homogeneous system
associated with (\ref{eq:l-s-sys}):
\begin{equation}\label{eq: mainsys}
y'=\left(\frac{A}{x}+B(x)\right)y.
\end{equation}

In the analytical theory of differential equations the structure of
the fundamental operator of the system $({\ref{eq: mainsys}})$ is
completely investigated under the assumption that the mapping
$B(\cdot)$ is holomorphic in the neighborhood of the singular point
$x=0$ (see e.g. \cite{KodLev58}). In the case where $B(\cdot)$ is
continuous only, the following proposition which is a simple
modification of the result stated in \cite[p.~275]{DalKre70} holds
true.

\begin{proposition}\label{pr:Fundmatr}
There exist numbers $x_0\in(0, \infty)$, $K>0$, and $r>0$ such that
the fundamental operator of the system (\ref{eq: mainsys}) admits
the representation in the form
\begin{equation}\label{fund} Y(x)=\left(E+U(x)\right)x^A,\quad x\in (0,x_0],
\end{equation}  where $E\in \mathrm{Hom}(\mathbb R^n)$ is a unit operator and the mapping
$U(\cdot)\in C^1\big((0,\infty) \mapsto $\\ $\mathrm{Hom}(\mathbb R^n)\big)$ satisfies the
estimate
\[
\|U(x)\|\leq K x|\ln x|^r, \quad x\in (0, x_0].
\]
\end{proposition}
\begin{proof}  The
mapping $Y(\cdot):(0,x_0] \mapsto \mathrm{Hom}(\mathbb R^n)$ defined
by $(\ref{fund})$ is a fundament\-al ope\-ra\-tor of the system
$(\ref{eq: mainsys})$ if $U(x)$ satisfies the equation
\[
U'=\frac{1}{x}\left(AU-UA\right)+B(x)\left(E+U\right), \quad x\in
(0,x_0].
\]
After the substitution $\displaystyle{x=e^{-t}}$ we obtain the
following equation  for the opera\-tor
$\displaystyle{V(t):=U(e^{-t})}$:
\begin{equation}\label{eq_v}
\dot{V}=VA-AV-e^{-t}B(e^{-t})\left(E+V\right).
\end{equation}
Thus we are to find the solution to this equation which satisfies
the inequality
\[
\|V(t)\|\leq Kt^r e^{-t}, \quad t\in [t_0,\infty)
\]
for a certain value of $t_0>0.$

The equation  (\ref{eq_v}) can be identified in $\mathbb R^{n^2}$
with the system of the form
\begin{equation}\label{eq_v_comp}
\dot{v}=\mathcal{A} v +e^{-t}\bigl(H(t)v +h(t)\bigr)
\end{equation}
where $\mathcal{A}\in \mathrm{Hom}(\mathbb R^{n^2})$ is a constant
operator and the mappings $H(\cdot)\in C\big([t_0,\infty) \mapsto
\mathrm{Hom}(\mathbb R^{n^2})\big)$ and $h(\cdot)\in C([t_0,\infty)
\mapsto \mathbb R^{n^2})$ satisfy the inequalities $\|H(t)\|\leq
M$,\, $\|h(t)\|\leq M$ for any $t\in [t_0,\infty).$

Now the required result can be obtained as an obvious consequence of
Lemma 1 and Lemma 2 stated below.
\end{proof}

\begin{lemma}\label{lem:bounsol}
Let $\mathcal{A}\in \mathrm{Hom}(\mathbb R^N))$. Then there exists a
mapping $G_\mathcal{A}(\cdot)\in C^{\infty}\big(\mathbb{R}\to
\mathrm{Hom} (\mathbb{R}^N)\big)$ such that for any function
$\mathfrak{f}(t)\in C([t_0,\infty)\to \mathbb{R}^N)$ satisfying the
estimate
\begin{equation} \label{lemma_est}
\|\mathfrak{f}(t)\|\leq M_{\mathfrak{f}} e^{-t}, \,t\in[t_0,\infty),
\end{equation}
with some constant $M_{\mathfrak{f}}>0$ the system
\begin{equation}\label{eq:syslemmabs}
\dot{y}=\mathcal{A}y+\mathfrak{f}(t)
\end{equation}
possesses a bounded on the semi-axis $[t_0, \infty)$ solution of the
form
\begin{equation*}
y(t)=\int\limits_{t_0}^{\infty}G_\mathcal{A}(t-s)\mathfrak{f}(s)ds.
\end{equation*}
This solution satisfies the inequality
\begin{equation*}
\|y(t)\|\leq C_\mathcal{A} M_\mathfrak{f}
e^{-t}\left(1+(t-t_0)^r\right)
\end{equation*}
where $C_\mathcal{A}$ is a positive constant depending on $\mathcal{A}$ only and
$r$ is the maximal dimension of Jordan blocks corresponding to
eigenvalues with the real part equal to $-1$ in the normal form
matrix of the operator $\mathcal{A}$.

If, in addition, $\mathfrak{f}(t)=o(e^{-t}),\;t\to \infty $, then the
solution $y(t)$ has the property ${y(t)=o\big(e^{-t}t^r\big),}$ ${t\to
\infty}$.
 \end{lemma}
\begin{proof}
We give the proof of the first part of the Proposition for the case
where $r\ge1$. Note that there exist three projectors
$\mathcal{P}_i:\mathbb{R}^N\to \mathbb{R}^N,\, i=\overline{1,3},$
such that ${\mathcal{P}_i \mathcal{P}_k=0,\, i\neq k,}$
${\mathcal{P}_1+\mathcal{P}_2+\mathcal{P}_3=E}$, and for some
constants $K_\mathcal{A}>0,\,\gamma_1>-1,\,$ $\gamma_2<-1$ the
following inequalities hold true:
\begin{gather*}
\|e^{\mathcal{A}\tau}\mathcal{P}_1\|\leq K_\mathcal{A} e^{\gamma_1 \tau},\,\tau\leq 0,\\
\|e^{\mathcal{A}\tau}\mathcal{P}_2\|\leq
K_\mathcal{A}\left(1+\tau^{r-1}\right) e^{-\tau},\,\tau\geq 0,\\
\|e^{\mathcal{A}\tau}\mathcal{P}_3\|\leq K_\mathcal{A} e^{\gamma_2
\tau},\,\tau\geq 0.
\end{gather*}

Now we define a function $G_\mathcal{A}(\tau)$ as follows:
\[
G_\mathcal{A}(\tau)=\left\{\begin{array}{cc}
                 -e^{\mathcal{A}\tau}\mathcal{P}_1, & \tau\leq 0, \\
                 e^{\mathcal{A}\tau}(\mathcal{P}_2+\mathcal{P}_3), & \tau>0.
               \end{array}
\right.
\]
The function
\begin{gather*}
y(t):=\int\limits_{t_0}^{\infty}G_\mathcal{A}(t-s)\mathfrak{f}(s)ds  \\ \equiv
\int\limits_{t_0}^{t}e^{\mathcal{A}(t-s)}\mathcal{P}_2
\mathfrak{f}(s)ds+\int\limits_{t_0}^{t}e^{\mathcal{A}(t-s)}\mathcal{P}_3
\mathfrak{f}(s)ds-\int\limits_{t}^{\infty}e^{\mathcal{A}(t-s)}\mathcal{P}_1
\mathfrak{f}(s)ds
\end{gather*}
is well defined and there exists a constant $C_\mathcal{A}
> 0$ dependent on the operator $\mathcal{A}$ only such that
 \begin{gather*}
   \|y(t)\| \leq
   K_\mathcal{A} M_\mathfrak{f} \Biggl(\int\limits_{t_0}^{t}\bigl(1+(t-s)^{r-1}\bigr)
   e^{-(t-s)}e^{-s}ds \\ +
   \int\limits_{t_0}^t
   e^{\gamma_2 (t-s)}e^{-s}ds+\int\limits_{t}^{\infty}e^{\gamma_1(t-s)}e^{-s}ds\Biggr)
   \\ \leq
    K_\mathcal{A} M_\mathfrak{f} e^{-t} \Biggl((t-t_0)+
   \frac{(t-t_0)^r}{r}+
   \frac{1-e^{-(|\gamma_2|-1)(t-t_0)}}{|\gamma_2|-1}+
   \frac{1}{\gamma_1+1}\Biggr)  \\ \leq
    C_\mathcal{A} M_\mathfrak{f}e^{-t}\bigl(1+(t-t_0)^r\bigr).
  \end{gather*}
Therefore, for $y(t)$ the inequality $(\ref{lemma_est})$ holds true.
One can easily make sure by the direct check that this function is in
fact the solution to the system $(\ref{eq:syslemmabs})$.

Now let  $\mathfrak{f}(t)=o(e^{-t}),\;t\to \infty $. Then for an
arbitrary $\epsilon
>0$ one can choose $T(\epsilon)>t_0$ in such a way that $\|\mathfrak{f}(t)\|\le \epsilon
e^{-t}$ for $t\ge T(\epsilon)$. Represent the solution $y(t)$ in the form
\begin{gather*}
  y(t)=\int_{t_0}^{T(\epsilon)}G_\mathcal{A}(t-s)\mathfrak{f}(s)\,ds+
  \int_{T(\epsilon)}^{\infty}G_\mathcal{A}(t-s)\mathfrak{f}(s)\,ds.
\end{gather*}
In accordance with what has been proved above, the norm of the second
addend does not exceed ${C_\mathcal{A}\epsilon e^{-t}(1+(t-T(\epsilon))^r)}$ for any
$t\ge T(\epsilon)$. For the first addend, when $t\ge T(\epsilon)$ we have:
\begin{gather*}
\int\limits_{t_0}^{T(\epsilon)}G_\mathcal{A}(t-s)\mathfrak{f}(s)\,ds
= \int\limits_{t_0}^{T(\epsilon)}e^{\mathcal{A}(t-s)}\mathcal{P}_2
\mathfrak{f}(s)ds+\int\limits_{t_0}^{T(\epsilon)}e^{\mathcal{A}(t-s)}\mathcal{P}_3
\mathfrak{f}(s)ds.
\end{gather*}
If $r=0$, then $\mathcal{P}_2=0$ and
\begin{gather*}
\Bigg\|\int\limits_{t_0}^{T(\epsilon)}G_\mathcal{A}(t-s)\mathfrak{f}(s)\,ds\Bigg\|=O\big(e^{\gamma_2t}\big)=o\big(e^{-t}\big),\quad
t\to \infty.
\end{gather*}
If $r>0$, then
\begin{gather*}
\Bigg\|\int\limits_{t_0}^{T(\epsilon)}G_\mathcal{A}(t-s)\mathfrak{f}(s)\,ds\Bigg\|
=O\big(e^{-t}\big)=o\big(e^{-t}t^r\big),\quad t\to \infty.
\end{gather*}
\end{proof}

\begin{lemma}\label{lem:bsv} Assume that
$H(\cdot)\in C\big([t_0,\infty) \mapsto \mathrm{Hom}(\mathbb R^{N})\big), \;h(\cdot)\in
C\big([t_0,\infty)$ $ \mapsto \mathbb R^{N})\big)$, and that there exist constants
$M>0,\;m>0$ such that $\|H(t)\|\leq M, \|h(t)\|\leq m$ for any $t\ge
t_0$. Let $C_\mathcal{A}$ and $r$ be the numbers defined in Lemma
\ref{lem:bounsol}. If the inequalities
\begin{equation}\label{eq:t_0:q<1}
t_0>r,\quad q:=2C_\mathcal{A}Mt_0^re^{-t_0}<1
\end{equation}
hold true, then the system (\ref{eq_v_comp}) has a solution $v(t)$
such that
\[
  \|v(t)\|\le \frac{2C_\mathcal{A}m}{1-q}t^re^{-t},\quad t\ge t_0.
\]
If, in addition, $h(t)\to 0,\;t\to \infty $, then
$v(t)=o(t^re^{-t}),\;t\to \infty $.
\end{lemma}
\begin{proof}

In view of the Lemma \ref{lem:bounsol}, we are going to find the
solution to the system $(\ref{eq_v_comp})$ satisfying the integral
equation
\begin{equation}\label{int_eq}
v(t)=\int\limits_{t_0}^{\infty}G_\mathcal{A}(t-s)e^{-s}\bigl(H(s)v(s)+h(s)\bigr)ds.
\end{equation}

Denote
$$\mathcal{G}[v(\cdot)](t):=\int\limits_{t_0}^{\infty}G_\mathcal{A}(t-s)e^{-s}\bigl(H(s)v(s)+h(s)\bigr)ds$$
and define the space of functions $$\mathcal{M}_{t_0,C}:=\{v(t)\in C([t_0,
\infty)\to \mathbb{R}^N): \|v(t)\|\leq C t^re^{-t},\,t\geq t_0\}.$$ Let
us show that if (\ref{eq:t_0:q<1}) holds true, then it is possible to
choose the constant
 $C>0$ in such a way that
${\mathcal{G}:\mathcal{M}_{t_0,C} \mapsto \mathcal{M}_{t_0,C}}$ and
this mapping is a contraction in the uniform metric.

 The Lemma \ref{lem:bounsol} implies that
\begin{gather*}
\|\mathcal{G}[v(\cdot)](t)\|\leq C_\mathcal{A} \bigl(M \sup\limits_{t\geq
t_0}\bigl(Ct^re^{-t}\bigr)+m\bigr)e^{-t}\bigl(1+(t-t_0)^r\bigr) \\
\leq 2C_\mathcal{A}\bigl(MCt_0^re^{-t_0}+m\bigr)t^re^{-t},\quad t_0>r,
\end{gather*}
for any function $v(t)\in \mathcal{M}_{t_0,C}$. Besides, when $t_0>r$, for
any $v(t),\,u(t)\in \mathcal{M}_{t_0,C}$ we obtain:
\begin{gather*}
\|\mathcal{G}[v(\cdot)-u(\cdot)](t)\|\leq C_\mathcal{A} M
e^{-t}\bigl(1+(t-t_0)^r\bigr)\sup\limits_{t\geq t_0}\|v(t)-u(t)\|
\\ \leq 2C_\mathcal{A}Mt^r_0e^{-t_0}\sup\limits_{t\geq
t_0}\|v(t)-u(t)\|=q\sup\limits_{t\geq t_0}\|v(t)-u(t)\|.
\end{gather*}
Since $q<1$, it is clear that $\mathcal{G}$ is a contraction mapping
on $\mathcal{M}_{t_0,C}$ once the following inequality  holds true:
\begin{gather*}
2C_\mathcal{A}(MCt^r_0e^{-t_0}+m)\le C.
\end{gather*}
Hence, by setting $$C:=\frac{\displaystyle 2C_\mathcal{A}m}{\displaystyle
1-q}$$ we guarantee the existence of a unique solution $v(t)\in
\mathcal{M}_{t_0,C}$ to the equation $(\ref{int_eq})$.

Now, suppose, in addition, that $h(t)\to 0,\; t\to \infty $. Since
the solution $v(t)$ can be represented in the form
\begin{gather*}
 v(t)= \int\limits_{t_0}^{\infty}G_\mathcal{A}(t-s)\mathfrak{f}(s)ds
\end{gather*}
where $\mathfrak{f}(t)=e^{-t}(H(t)v(t)+h(t))=o(e^{-t}),\;t\to \infty $,
then in accordance with the Lemma \ref{lem:bounsol} we obtain:
${v(t)=o(t^re^{-t}),\;t\to \infty}$.
\end{proof}

\section{Additional conditions for the linear\\ homogeneous system and
their corollaries}\label{sec:AddCond}

Hereafter we assume that for the linear homogeneous system (\ref{eq:
mainsys}) conditi\-ons \textbf{A, B} described below hold true.
These conditi\-ons concern local properties of the system in
neighborhoods of the points $x=0$ and $x=+\infty $.

\textbf{A}: the characteristic polynomial of the operator
 $A$ has no roots with real part equal to 1;

\textbf{B}: the system  (\ref{eq: mainsys}) is exponentially
dichotomic on the semi-axis $[x_0,\infty)$ for some (and therefore, for
any) positive $x_0$.

Let $y(x,y_0)$ be a solution to the system (\ref{eq: mainsys})
satisfying the initial condition $y(x_0,y_0)$ $=y_0$. For the sake of
generality we assume that the characteristic polynomial of the
operator $A$ has roots with real parts both less and greater than 1
and the system (\ref{eq: mainsys}) has both bounded and unbounded
solutions  on the half-line $[x_0,\infty)$.

Under the conditions \textbf{ A} and \textbf{B} there exist subspaces
$\mathbb V_+$ and $\mathbb U_-$ with the following properties:\par 1. There exists
$\alpha >0$ such that for any subspace $\mathbb V_-$ which is a direct
supplement of $\mathbb V_+$ to $\mathbb R^n$ one can choose a constant $c_0>0$ in
such a way that
\begin{align}
   \|y(x,y_0)\| \le
  c_0\left(\frac{x}{s}\right)^{1+\alpha}\|y(s,y_0)\|,& \quad 0<x\le s\le x_0,
\quad \text{if}\;\;y_0\in \mathbb V_+;
  \label{eq:V+}\\
    \|y(x,y_0)\| \le
  c_0\left(\frac{x}{s}\right)^{1-\alpha}\|y(s,y_0)\|,&\quad 0<s\le x\le x_0,
  \quad \text{if}\;\;y_0\in \mathbb V_- \label{eq:V-}.
\end{align}
(This property results from the Proposition~\ref{pr:Fundmatr} and the
condition \textbf{A}.)
\par 2. There exists
a constant $\gamma >0$ such that for any subspace $\mathbb U_+$ which is a
direct supplement of $\mathbb U_-$ to $\mathbb R^n$ one can choose a constant
$c_*>0$ in such a way that
\begin{align}
  \|y(x,y_0)\|\le c_*e^{-\gamma(x-s)}\|y(s,y_0)\|,&\quad x_0\le s\le x,
  \quad \text{if}\;\;y_0\in \mathbb U_- \label{eq:U-};\\
  \|y(x,y_0)\|\le c_*e^{\gamma(x-s)}\|y(s,y_0)\|,&\quad x_0\le x\le
  s,\quad \text{if}\;\;y_0\in \mathbb U_+ \label{eq:U+}.
\end{align}
(See remark 3.4 in \cite[p.~235]{DalKre70}.)
\par If the subspace $\ker A$ is non-trivial, then there exists a subspace $\mathbb V_-^0$
iso\-morphic to the subspace $\ker A$ and having the next property:
\par 3. For any $y_*\in \mathbb V_-^0$ there exists a unique vector
$\zeta \in \ker A$ such that
\begin{gather}\label{eq:V0}
y(x,y_*)=\big(E+\Theta (x)\big)\zeta,\quad x\to +0,
\end{gather}
where $\Theta(\cdot)\in C^1\big([0,x_0] \mapsto \mathrm{Hom}(\mathbb
R^n)\big)$ and $\Theta(x)=x(E-A)^{-1}B(0)+o(x),\;x\to +0$. At the
same time, $\mathbb V_-^0\cap \mathbb V_+=\{0\}$ and the subspace
$\mathbb V_+\oplus \mathbb V_-^0$ coincides with the subspace of
initial values (for  $x=x_0$) of continuously differentiable on
 $[0,\infty)$ solutions to the system ~(\ref{eq: mainsys}). (See the corollary from the Proposition
\ref{exC^1sol} which is stated in section~\ref{sec:MainRes}.)

Now the space  $\mathbb R^n$ can be represented as the direct sum of six
subspaces $\mathbb L_1,\ldots,\mathbb L_6$ defined in the following way:

1) $\mathbb L_1:=\mathbb U_-\cap \mathbb V_+$;

2) $\mathbb L_2$ is a direct supplement of the subspace $\mathbb L_1$ to
$\mathbb U_-\cap (\mathbb V_+ \oplus \mathbb V_-^0)$, so that
\begin{gather*}
  \mathbb L_1\oplus\mathbb L_2=\mathbb U_-\cap (\mathbb V_+ \oplus \mathbb V_-^0);
\end{gather*}

 3) $\mathbb L_3$ is a direct supplement of the subspace $\mathbb U_-\cap (\mathbb V_+ \oplus \mathbb V_-^0)$
 to $\mathbb U_-$, so that
\begin{gather*}
  \mathbb L_1\oplus \mathbb L_2 \oplus \mathbb L_3=\mathbb U_-;
\end{gather*}

4) $\mathbb L_4$ is a direct supplement of the subspace $\mathbb
L_1=\mathbb U_-\cap \mathbb V_+$ to $\mathbb V_+$, so that
\begin{gather*}
  \mathbb V_+=\mathbb L_1\oplus\mathbb L_4;
\end{gather*}

5) $\mathbb L_5$ is a direct supplement of the subspace $(\mathbb U_-\cap
(\mathbb V_+ \oplus \mathbb V_-^0))\oplus \mathbb L_4$ to $\mathbb V_+ \oplus \mathbb V_-^0$, so that
\begin{gather*}
  \mathbb L_1\oplus \mathbb L_2\oplus \mathbb L_4\oplus \mathbb L_5=\mathbb V_+\oplus \mathbb V_-^0,
   \end{gather*}
and taking into account the equalities $(\mathbb L_1\oplus\mathbb
L_4)\cap\mathbb V_-^0=\{0\}$ and $\mathrm{dim}\,\mathbb
L_2+\mathrm{\dim}\,\mathbb L_5=\mathrm{dim}\,\mathbb V_-^0$ we
choose $\mathbb L_5\subset \mathbb V_{-}^0$;

6) $\mathbb L_6$ is a direct supplement of the subspace $\mathbb L_1\oplus \cdots
\oplus \mathbb L_5=\mathbb U_{-}\oplus \mathbb L_4\oplus \mathbb L_5$ to $\mathbb R^n$.

If the two subspaces $\mathbb U_+$ and $\mathbb V_-$, which are direct supplements
of the subspaces $\mathbb U_-$ and $\mathbb V_+$ respectively, are defined by the
equalities
\begin{gather*}
  \mathbb U_+:=\mathbb L_4\oplus \mathbb L_5\oplus\mathbb L_6,\quad \mathbb V_-:=\mathbb L_2\oplus
  \mathbb L_3\oplus \mathbb L_5\oplus\mathbb L_6,
\end{gather*}
then the above assumptions allow us to distinguish six types of
solu\-tions to the system (\ref{eq: mainsys}). Namely: if $y_0\in
\mathbb L_1$, then the solution $y(x,y_0)$ satisfies the
inequalities (\ref{eq:V+}) and (\ref{eq:U-}); the solution for which
 $y_0\in \mathbb L_2$ fulfills the inequality
(\ref{eq:U-}), and there exists unique $y_*\in \mathbb V_-^0$ such
that
\[
  \|y(x,y_0)-y(x,y_*)\|=o(x),\quad x\to 0;
\]
the solution for which $y_0\in \mathbb L_3$ satisfies the
inequalities (\ref{eq:V-}) and (\ref{eq:U-}), besides, for this
solution the derivative $y'(+0;y_0)$ does not exist; for the
solution with $y_0\in \mathbb L_4$ the inequalities (\ref{eq:V+})
and (\ref{eq:U+}) hold true; the solution having initial value from
$\mathbb L_5$ fulfills the inequality (\ref{eq:U+}), and there is a
unique $\zeta \in \ker A$ for which (\ref{eq:V0}) is valid; finally,
if $y_0\in \mathbb L_6$, then the solution $y(x,y_0)$ satisfies
inequalities (\ref{eq:V-}) and (\ref{eq:U+}), and  for such a
solution the derivative $y'(+0;y_0)$ does not
exist.\label{page:vlast2}

Let $E=P_1+\cdots +P_6$ be the decomposition of the unit operator into
the sum of mutually disjunctive projectors generated by the
decompo\-si\-tion $\mathbb R^n=\mathbb L_1\oplus \cdots \oplus \mathbb L_6$. Define the
following operators:
\begin{gather*}
  Q_+:=P_1+P_4,\;Q_-:=P_2+P_3+P_5+P_6,\\
  P_-:=P_1+P_2+P_3,\;P_+:=P_4+P_5+P_6.
\end{gather*}
It is clear that the  projectors $Q_+,\;Q_-$ correspond to the
decomposition $\mathbb R^n=\mathbb V_+\oplus \mathbb V_-$, while $P_-,\; P_+$  correspond
to the decomposition $\mathbb R^n=\mathbb U_-\oplus \mathbb U_+$, and there exist
constants $C_0>0$ and $C_*>0$ such that for the normalized at the
point $x_0$ evolution operator $Y(x;x_0)$ of the system (\ref{eq:
mainsys}) the following estimates are valid:
\begin{gather}
\|Y(x;x_0)Q_+Y^{-1}(s;x_0)\|\leq
C_0\left(\frac{x}{s}\right)^{1+\alpha},\quad
0<x\leq s\leq x_0, \label{eq:V+Y}\\
 \|Y(x;x_0)Q_-Y^{-1}(s;x_0)\|\leq
C_0\left(\frac{x}{s}\right)^{1-\alpha},\quad 0<s\leq x\leq
x_0,\label{eq:V-Y}
\end{gather}
and
\begin{gather}\label{eq:P-Y}
\|Y(x;x_0)P_-Y^{-1}(s;x_0)\|\leq C_* e^{-\gamma (x-s)},\quad x_0\leq
s\leq x,\\ \|Y(x;x_0)P_+Y^{-1}(s;x_0)\|\leq C_* e^{-\gamma (s-x)},\quad
x_0\leq x\leq s.\label{eq:P+Y}
\end{gather}

\medskip

\section{Generalized Green function for the boundary value problem
with homogeneous boundary\\ conditions}\label{sec:GenGreenFunc}

Consider the boundary value problem of the form
\begin{gather}\label{eq:nonhom1}
y'=\left(\frac{A}{x}+B(x)\right)y+g(x),\\ \label{eq:hombc}
  y(\cdot)\in C^1([0,\infty) \mapsto \mathbb R^n),\quad y(+0)=0,\quad y(+\infty)=0,
\end{gather}
in the case of  function $g(\cdot)\in C([0, \infty) \mapsto \mathbb R^n)$ vanishing at
infinity: $g(x)\to 0$ when $x\to +\infty$. Let $m:=\sup_{x\in
[0,\infty)}\|g(x)\|$.

First, we prove that any element of $\ker A$ can be brought into
 corres\-pon\-dence with at least one solution which is continuously
 differentiable on $[0,\infty)$.
\begin{proposition}\label{exC^1sol}
Under the  condition $\mathbf{A}$, for any $\zeta \in \ker A$ there exists a
solution to the system (\ref{eq:nonhom1}) of the form
\begin{equation}\label{eq:diffsol}
  y_{\zeta}(x)=\zeta +\zeta_1x+o(x),\quad x\to +0,
\end{equation}
where  $\zeta_1:=(E-A)^{-1}\bigl(B(0)\zeta+g(0)\bigr)$. Conversely, every
continuously differen\-tiable on $[0,\infty)$ solution to the system
(\ref{eq:nonhom1}) can be represented in the form (\ref{eq:diffsol}).
\end{proposition}
\begin{proof} The change of dependent variable $y=\zeta +\zeta_1x+z$ in
(\ref{eq:nonhom1}) leads to the system
\begin{gather*}
  z'=\left(\frac{A}{x}+B(x)\right)z+\tilde{g}(x)
\end{gather*}
where $\tilde{g}(x)=(B(x)-B(0))\zeta+g(x)-g(0)+xB(x)\zeta_1=o(1),\;x\to +0$.
After the substitution $x=e^{-t}$ we obtain the system
\begin{gather}\label{eq:sysforz}
  \dot{z}=-\big(A+e^{-t}B\big(e^{-t}\big)\big)z-e^{-t}\tilde{g}\big(e^{-t}\big).
\end{gather}
The value $t_0>0$ can be chosen sufficiently large, so that the
conditions of Lemma 2 hold true for this system. In accordance with
this Lemma and taking into account that the characteristic polynomial
of the opera\-tor $-A$ has no roots with the real part equal to $-1$,
there exists the solution $\tilde{z}(t)$ to the system (\ref{eq:sysforz})
satisfying the equality
\begin{gather*}
  \tilde{z}(t)=-\int_{t_0}^{\infty}G_{-A}(t-s)e^{-s}\big(B\big(e^{-s}\big)\tilde{z}(s)+
  \tilde{g}(e^{-s})\big)\,ds
\end{gather*}
and, besides, having the property $\tilde{z}(t)=o(e^{-t}),\;t\to \infty $.
But in such a case the function ${z(x):=\tilde{z}(-\ln x)=o(x),\;x\to 0}$,
generates the required solution $y(x)=\zeta+\zeta_1x+z(x)$ of the system
(\ref{eq:nonhom1}). The second part of the Proposition is obvious.
\end{proof}

\begin{corollary}\label{cor:exC^1sol} There exists a mapping
$\Theta(\cdot)\in C^1\big([0,x_0] \mapsto \mathrm{Hom}(\mathbb R^{n})\big)$ of the
form $\Theta(x)= x(E-A)^{-1}B(0)+o(x),\;x\to +0$, such that for any $\zeta
\in \ker A$ the function
\begin{gather*}
  y_\zeta(x)=\big(E+\Theta (x)\big)\zeta
\end{gather*}
is a solution to the homogeneous system (\ref{eq: mainsys})
corresponding to the vector $\zeta $.
\end{corollary}

\begin{proposition}\label{pr:c^1sol-s} The family of functions defined as
\begin{gather}\nonumber
\bar{y}_v(x)=Y(x;x_0)v
+\int\limits_{0}^{x}Y(x;x_0)Q_-Y^{-1}(s;x_0)g(s)ds \\ +
\int\limits_{x_0}^{x}Y(x;x_0)Q_+Y^{-1}(s;x_0)g(s)ds,
\label{eq:C^1(0,8)}
\end{gather}
where $v $ is an arbitrary vector from $\mathbb V_+\oplus\mathbb V_-^0$, determines
all solutions to the system (\ref{eq:nonhom1}) of the class
$C^1([0,\infty) \mapsto \mathbb R^n)$. Each of such solutions satisfies the
condition $\bar{y}_v(+0)=0$ iff $v\in \mathbb L_1\oplus\mathbb L_4=\mathbb V_+$.
\end{proposition}
\begin{proof}
In view of the estimates (\ref{eq:V+Y}), (\ref{eq:V-Y}), for any
$x\in[0,x_0)$ the integrals in the formula (\ref{eq:C^1(0,8)})
satisfy the inequalities
\begin{gather*}
\left\|\int\limits_{0}^{x}\!Y(x;x_0)Q_-Y^{-1}(s;x_0)g(s)ds\right\|\leq
m{C}_0 x^{1-\alpha}\int\limits_{0}^{x}s^{\alpha -1}ds =m {C}_0
\frac{x}{\alpha};\\
\left\|\int\limits_{x}^{x_0}Y(x;x_0)Q_+Y^{-1}(s;x_0)g(s)ds \right\|\leq m
{C}_0 x^{1+\alpha}\int\limits_{x}^{x_0}s^{-1-\alpha}ds  \\ \le m {C}_0
\frac{x^{1+\alpha}(x^{-\alpha}-x_0^{-\alpha})}{\alpha}\leq \frac{m {C}_0 x}{\alpha}.
\end{gather*}
By means of direct check, one can easily verify that each function of
the set (\ref{eq:C^1(0,8)}) is a solution to the system
(\ref{eq:nonhom1}). From the definition of $\mathbb V_+,\;\mathbb V_-^0$,
properties of the spaces $\mathbb L_1,\mathbb L_2,\mathbb L_4,\mathbb L_5$ (see
p.~\pageref{page:vlast2}) it follows that for any $v\in \mathbb V_+\oplus
\mathbb V_-^0$ there exists a limit $\lim_{x\to +0}Y(x;x_0)v=:\zeta(v)\in \ker
A$, and $Y(x;x_0)v$ $=\zeta(v)+O(x),\;x\to 0$. Therefore, $\bar{y}_{v}
(x)=\zeta(v)+O(x),\;x\to 0$, and the difference
$\bar{y}_v(x)-y_{\zeta(v)}(x)$, where $y_{\zeta(v)}(x)$ is the solution from
the Proposition~\ref{exC^1sol}, is a solution to the system (\ref{eq:
mainsys}). Moreover, $\|\bar{y}_v(x)-y_{\zeta(v)}(x)\|=O(x),\;x\to 0$. This
implies that $\|\bar{y}_v(x)-y_{\zeta(v)} (x)\|=o(x),\;x\to 0$, and thus,
$\bar{y}_v(x_0)-y_{\zeta(v)} (x_0)\in \mathbb L_1\oplus\mathbb L_4$. Taking into account
the Proposition~\ref{pr:c^1sol-s}, we can conclude that
$\bar{y}_v(x)\in C^1([0,x_0] \mapsto \mathbb R^n)$.

Since each non-trivial solution to the system (\ref{eq: mainsys})
with the initial condition $y_0\in \mathbb L_2\oplus\mathbb L_5$ has a non-zero
limit when $x\to +0$, the equality $\bar{y}_v(+0)=0$ is equivalent to
$v\in \mathbb L_1\oplus\mathbb L_4$.
\end{proof}

 It is well known  (see e.g. \cite{DalKre70}) that all solutions to the
 system  (\ref{eq:nonhom1}) which are bounded on the semi-axis $[x_0,\infty)$
 form a family
\[
  \begin{split}
\hat{y}_{u}(x)=&Y(x;x_0)u
+\int\limits_{x_0}^{x}Y(x;x_0)P_-Y^{-1}(s;x_0)g(s)ds\\ -
&\int\limits_{x}^{\infty}Y(x;x_0)P_+Y^{-1}(s;x_0)g(s)ds
  \end{split}
\]
where $u $ is an arbitrary vector from $\mathbb U_-$.

It is also known that the following proposition holds true:
\begin{proposition}\label{pr:yto0} If $g(x)\to 0,\;x\to \infty $,
then $\hat{y}_{u}(x) \to 0,\;x\to \infty.$\end{proposition}
\begin{proof} For the sake of completeness let us sketch the proof.

 For an arbitrary $\epsilon >0$ let choose the value $x(\epsilon)>x_0$ in such a way that
$\|g(x)\|<\epsilon $ for any $x>x(\epsilon)$. Then for
$x>x(\epsilon)$ we have
\begin{gather*}
  \hat{y}_u (x)=Y(x;x_0)u +\int_{x_0}^{x(\epsilon)}Y(x;x_0)P_-
  Y^{-1}(s;x_0)g(s)\,ds\\
  +\int_{x(\epsilon)}^{x}Y(x;x_0)P_-Y^{-1}(s;x_0)g(s)\,ds+
  \int_{x}^{\infty}Y(x;x_0)P_+Y^{-1}(s;x_0)g(s)\,ds.
  \end{gather*}
The first addend in this expression tends to zero when $x\to \infty
$, norm of each of the last two addends does not exceed $\epsilon
K/\gamma $, and for the second addend it holds
\begin{gather*}
  \left\|\int_{x_0}^{x(\epsilon)}Y(x;x_0)P_-Y^{-1}(s;x_0)g(s)\,ds\right\|=O(e^{-\gamma x}),\quad x\to \infty.
\end{gather*}
\end{proof}

Now,  to find all  solutions to the system (\ref{eq:nonhom1}) which
satisfy the conditions (\ref{eq:hombc}) we bind parameters
$v\in\mathbb L_1\oplus\mathbb L_4 $ and $u\in \mathbb
L_1\oplus\mathbb L_2\oplus\mathbb L_3$ by means of the equality
$\bar{y}_{v}(x_0)=\hat{y}_{u}(x_0)$ which can be rewritten in the
form
\begin{gather*}
  P_-u -Q_+v
  =\int_{0}^{x_0}Q_-Y^{-1}(s;x_0)g(s)\,ds+\int_{x_0}^{\infty}P_+Y^{-1}(s;x_0)g(s)\,ds
\end{gather*}
or, equivalently,
\begin{gather*}
  (P_1+P_2+P_3)u -(P_1+P_4)v \\ =
  \int_{0}^{x_0}(P_2+P_3+P_5+P_6)Y^{-1}(s;x_0)g(s)\,ds \\
  +\int_{x_0}^{\infty}(P_4+P_5+P_6)Y^{-1}(s;x_0)g(s)\,ds.
\end{gather*}
From this it follows that
\begin{gather*}
  P_1u =P_1v,\quad P_2u =\int_{0}^{x_0}P_2Y^{-1}(s;x_0)g(s)\,ds,\\
  P_3u =\int_{0}^{x_0}P_3Y^{-1}(s;x_0)g(s)\,ds,\quad
   P_4v =-\int_{x_0}^{\infty}P_4Y^{-1}(s;x_0)g(s)\,ds,
   \end{gather*}
and the function $g(x)$ must satisfy the additional condition
\begin{equation}\label{eq:ortcond}
  \int_{0}^{\infty}(P_5+P_6)Y^{-1}(s;x_0)g(s)\,ds=0.
\end{equation}

Therefore, if the condition (\ref{eq:ortcond}) holds true, the
solutions to the problem (\ref{eq:nonhom1})--(\ref{eq:hombc}) can be
given by the formula
\begin{equation}\label{eq:famrbs}
\begin{split}
  y&=Y(x;x_0)\biggl(P_1v +\int_{x_0}^{x}P_1Y^{-1}(s;x_0)g(s)\,ds\\
  &+\int_{0}^{x}(P_2+P_3)Y^{-1}(s;x_0)g(s)\,ds\\ &-
  \int_{x}^{\infty}(P_4+P_5+P_6)Y^{-1}(s;x_0)g(s)\,ds\biggr).
 \end{split}
\end{equation}
This formula can also be rewritten in the following way:
\[
  \begin{split}
  y&=Y(x;x_0)\biggl(P_1v-\int_{x_0}^{\infty}P_4Y^{-1}(s;x_0)g(s)\,ds
  \\&+\int_{x_0}^{x}Q_+Y^{-1}(s;x_0)g(s)\,ds
  +\int_{0}^{x}Q_-Y^{-1}(s;x_0)g(s)\,ds\biggr).
 \end{split}
\]

 Having defined the sets
\begin{gather*}
  D:=\{(x,s):0< x< s< x_0\}\cup
    \{(x,s):x_0\le s\le x  \},\\ D_+:=\{(x,s):0<s\le x\},
    \quad D_-:=\{(x,s):0<x<s\}
\end{gather*}
and the functions
\begin{align*} G_1(x,s):=&
  \begin{cases}
    \phantom{-}Y(x;x_0)P_1Y^{-1}(s;x_0), & (x,s)\in D\cap D_+, \\
    -Y(x;x_0)P_1Y^{-1}(s;x_0), & (x,s)\in D\cap D_-,\\
    \phantom{-}0, & (x,s)\in (D_+\cup D_-)\setminus D,
  \end{cases}\\
 G_2(x,s):=&
\begin{cases}
    \phantom{-}Y(x;x_0)(P_2+P_3)Y^{-1}(s;x_0), & (x,s)\in D_+,\\
    -Y(x;x_0)(P_4+P_5+P_6)Y^{-1}(s;x_0), & (x,s)\in D_-,
  \end{cases}\\
  G(x,s):=&G_1(x,s)+G_2(x,s),
  \end{align*}
and taking into account the formula (\ref{eq:famrbs}) we get the
following result.
\begin{proposition}\label{pr:famrbs}
There exists a solution to the boundary value problem
(\ref{eq:nonhom1})--(\ref{eq:hombc}) iff the condition
(\ref{eq:ortcond}) holds true, and in this case all soluti\-ons to
the problem are defined by the formula
\begin{gather*}
  y=Y(x;x_0)v +\int_{0}^{\infty}G(x,s)g(s)\,ds,\quad \forall v \in \mathbb L_1.
\end{gather*}
\end{proposition}

Now we are going to interpret the condition (\ref{eq:ortcond}) in
terms of soluti\-ons to the adjoint (with respect to the scalar
product $\langle \cdot,\cdot \rangle $) homo\-ge\-ne\-ous system
\begin{equation}\label{eq:conjsys}
\eta '=-\Big(\frac{A^*}{x}+B^*(x)\Big)\eta.
\end{equation}
Let $\eta(x,\eta_0)$ denote the solution to this system satisfying
the initial condition $\eta(x_0,\eta_0)$ $=\eta_0$. In what follows,
without loss of generality we assume that the scalar product in
$\mathbb R^n$ is determined in such a way that
$P_j^*=P_j,\;j=1,\ldots,6$.

Let $L_1([0,\infty) \mapsto \mathbb R^n)$ be the space of functions $f(\cdot):[0,\infty)
\mapsto \mathbb R^n$ for which $\int_{0}^{\infty}\|f(x)\|\,dx<\infty $.
\begin{proposition}\label{pr:intsol}
The solution $\eta(x,\eta_0)$ belongs to $L_1\big([0,\infty) \mapsto \mathbb R^n\big)$
iff $\eta_0\in \mathbb L_5\oplus\mathbb L_6$.
\end{proposition}
\begin{proof}  As is well known, $[Y^{-1}(x;x_0)\big]^*$ is a
fundamental operator of the adjoint system normalized at the point
$x_0$, and $$\langle \eta(x,\eta_0),y(x,y_0)\rangle \equiv\langle
\eta_0,y_0\rangle.$$ Let $y_0:=(P_1+\cdots +P_4)\eta_0\ne0$. If, in
addition, we suppose that $Q_+y_0\ne 0$, then in view of
\eqref{eq:V+Y}
\begin{gather*}
 \|Q_+y_0\|^2\le \|y(x,Q_+ y_0)\|\|\eta(x,\eta_0)\|\le
  c_0(x/x_0)^{1+\alpha}\|Q_+ y_0\|\|\eta(x,\eta_0)\|\quad
  \end{gather*}
for all $x\in (0,x_0]$, and thus $\|\eta(x,\eta_0)\|\ge \|Q_+
y_0\|(x/x_0)^{-1-\alpha}/c_0$ when $x\in (0,x_0]$. This implies that
$\eta(x,\eta_0)\not\in L_1([0,\infty) \mapsto \mathbb R^n)$.

If $Q_+y_0=0,$ then $y_0=(P_2+P_3)y_0=P_-y_0\ne0$. Hence, in view of
\eqref{eq:P-Y}
\begin{gather*}
\|P_-y_0\|^2\le \|y(x,P_-y_0)\|\|\eta(x,\eta_0)\|
 \le c_* e^{-\gamma (x-x_0)}\|P_-y_0\|\|\eta(x,\eta_0)\|
\end{gather*}
for all $x\ge x_0$. This also implies that  $\eta(x,\eta_0)\not\in
L_1([0,\infty) \mapsto \mathbb R^n)$.

On the other hand, if $y_0=0$, then $\eta_0=(P_5+P_6)\eta_0$. Now from
the inequalities (\ref{eq:V-Y}) and (\ref{eq:P+Y}) it follows that
\begin{gather*}
  \|\eta(x,\eta_0)\|\le \big\|\big[Y^{-1}(x;x_0)\big]^*(P_5+P_6)\big\|\|\eta_0\|
  \leq
C_0\left(\frac{x_0}{x}\right)^{1-\alpha}\|\eta_0\|
\end{gather*}
when $x\in (0, x_0]$, and
\begin{gather*}
\|\eta(x,\eta_0)\|\le
\big\|\big[Y^{-1}(x;x_0)\big]^*(P_5+P_6)\big\|\|\eta_0\|\leq C_* e^{-\gamma
(x-x_0)}\|\eta_0\|
\end{gather*}
when $x\geq x_0$.  Hence, $\eta(x,\eta_0)$ belongs to $L_1\big([0,\infty)
\mapsto \mathbb R^n\big)$.
\end{proof}

The above proof has the following  corollary:
\begin{proposition}\label{pr:ortcond}
The condition (\ref{eq:ortcond}) holds true iff the function $g(x)$
is ortho\-gonal (in sense of the scalar product $\langle \cdot,\cdot
\rangle_{L_2}:= \int_{0}^{\infty}\langle \cdot,\cdot \rangle \,dx$)
to each solution of the adjoint system (\ref{eq:conjsys}) belonging
to the space $L_1([0,\infty) \mapsto \mathbb R^n)$.
\end{proposition}

Now, let us show that the problem
(\ref{eq:nonhom1})--(\ref{eq:hombc}) has a generalized Green
function $\mathfrak G(x,s)$ defined by the following properties:

1. For any $s>0$ and $x\in [0,\infty)\setminus \{s\}$, it holds
\[
  \mathfrak L\mathfrak G(x,s)=-F(x;x_0)\Pi Y^{-1}(s;x_0)
\]
where $\mathfrak L:=\frac{d}{dx}-\big(\frac{A}{x}+B(x)\big)$, $\Pi :=P_5+P_6$,
and $F(\cdot,x_0)\in C\bigl([0,\infty)$ $\mapsto$ $ \mathrm{Hom}(\mathbb R^n)\bigr)$ is a
bounded mapping with the "biorthonormality" proper\-ty with respect
to the space of solutions of the adjoint system which belong to
$L_1([0,\infty) \mapsto \mathbb R^n)$:
\begin{gather*}
  \int_{0}^{\infty}\Pi Y^{-1}(x;x_0)F(x;x_0)\,dx=\Pi.
\end{gather*}
E.g., we may set
\begin{gather*}
  F(x;x_0):=\frac{\kappa^{1+\beta}x^\beta}{\Gamma(1+\beta)e^{\kappa  x}} Y(x;x_0)\Pi
  \end{gather*}
where $\kappa$ is an arbitrary number greater than $\gamma $ and
$\beta > 0$ is an arbitrary number with the property that all real
parts of eigenvalues of the matrix $A$ exceed $-\beta $. Obviously,
$F(+0;x_0)=F(+\infty;x_0)=0$.

 2. For any $x>0$ the unit jump property is valid:
$\mathfrak G(x+0,x)-\mathfrak G(x-0,x)=E$.

3. The condition of orthogonality to the space of solutions to the
corres\-ponding homogeneous boundary value problem is fulfilled:
$$\int_{0}^{\infty}P_1Y^*(x;x_0)\mathfrak G(x,s)\,dx=0.$$

4. For any $s>0$ the boundary conditions $\mathfrak
G(+0,s)=\mathfrak G(+\infty,s)=0$ are satisfied.

5. For any $g(\cdot)\in C([0,\infty) \mapsto \mathbb R^n)$ satisfying
\eqref{eq:ortcond}, the boundedness condition holds true:
$$\sup_{x\in [0,\infty)}\int_{0}^{\infty}\|\mathfrak G(x,s)g(s)\|\,ds <\infty. $$

Observe that the operator differential equation  $\mathfrak
LY=-F(x;x_0)$ has a particular solution
\begin{gather*}
Y=N(x;x_0):=-\int_{0}^{\infty}G(x,t)F(t;x_0)\,dt
\end{gather*}
which can be represented in the form
\begin{gather*}
 N(x;x_0)= Y(x;x_0)
\bigg(\Pi-\int_{0}^{x}\frac{\kappa^{1+\beta}t^\beta}{\Gamma(1+\beta)e^{\kappa t}}\,dt \Pi
\bigg).
\end{gather*}
(Note that generally $N(x;x_0)$ is unbounded on $(0,x_0)$, but it
vanishes at infinity.)

It is easily seen that the conditions 1--4 hold true for the operator
\begin{gather}\label{eq:gengrfunc}
\mathfrak G(x,s):=G(x,s)+Y(x;x_0)P_1M(s;x_0)+N(x;x_0)\Pi
Y^{-1}(s;x_0)
\end{gather}
once we set
\begin{gather*}
  M(s;x_0):= -
  \bigg[\int_{0}^{\infty}\!\!\!P_1Y^*(x;x_0)Y(x;x_0)P_1\,dx \bigg|_{\mathbb L_1}
  \bigg]^{-1}
\cdot \\ \cdot \int_{0}^{\infty}\!\!\!P_1Y^*(x;x_0)\big(
G(x,s)+N(x;x_0)\Pi Y^{-1}(s;x_0)\big)\,dx.
\end{gather*}
To show that the condition 5 is fulfilled it remains only to verify
that $M(s;x_0)$ is absolutely integrable on $[0,\infty)$. This property
 can be easily obtained from the next six estimates for the function
$$J(s;x;x_0):=\big\|P_1Y^*(x;x_0)(G(x,s)+N(x;x_0)\Pi
Y^{-1}(s;x_0))\big\|$$ which are based on inequalities
\eqref{eq:V+Y}--\eqref{eq:P+Y}.

1) Let $x<s<x_0.$ In this case $G(x,s) =
-Y(x;x_0)(P_1+P_4+\Pi)Y^{-1}(s;x_0),$ and therefore there exits a
constant $C_1(x_0)>0$ such that
\begin{gather*}
J(s;x;x_0)  \leq
\|Y(x;x_0)P_1\|\bigl(\|Y(x;x_0)(P_1+P_4)Y^{-1}(s;x_0)\| \\ + (\kappa
x)^{1+\beta}\|Y(x;x_0)\Pi Y^{-1}(s;x_0)\|\bigr)
\\ \leq
C_1(x_0)x^{1+\alpha}\bigl(({x}/{s})^{1+\alpha}+ x s^{\alpha-1}\bigr)\le
C_1(x_0)x^{1+\alpha}(1+s^{\alpha}).\end{gather*}

2) Let $s\leq x<x_0.$ Now $G(x,s) = Y(x;x_0)(P_2+P_3)Y^{-1}(s;x_0)$,
and there exists a constant $C_2(x_0)>0$ such that
\begin{gather*}
J(s;x;x_0) \leq \|Y(x;x_0)P_1\|\bigl(\|Y(x;x_0)Q_-Y^{-1}(s;x_0)\|\\ +
(\kappa x)^{1+\beta}\|Y(x;x_0)\Pi Y^{-1}(s;x_0)\|\bigr) \leq
C_2(x_0)x^2s^{\alpha-1}.
\end{gather*}

3) Let $s<x_0\le x.$ Now $G(x,s) = Y(x;x_0)(P_2+P_3)Y^{-1}(s;x_0)$,
hence,
 \begin{gather*} J(s;x;x_0) \leq
\|Y(x;x_0)P_1\|\bigl(\|Y(x;x_0)(P_2+P_3)Y^{-1}(s;x_0)\|\\ +
\|N(x;x_0)\Pi Y^{-1}(s;x_0)\|\bigr)  \\ \leq
C_0C_*e^{-\gamma(x-x_0)}\left(\frac{x_0}{s}\right)^{1-\alpha}\bigl(C_*+
\sup_{x\in[x_0,\infty)}\|N(x;x_0)\|\bigr)\\ \le C_3(x_0)e^{-\gamma x}s^{\alpha
-1}
\end{gather*}
for some constant $C_3(x_0)>0$.

 4) Let $x< x_0\le s.$ Since $G(x,s) =
-Y(x;x_0)(P_4+\Pi)Y^{-1}(s;x_0),$ then there exists a constant
$C_4(x_0)>0$ such that
\begin{gather*}
J(s;x;x_0)\leq \|Y(x;x_0)P_1\|\bigl(\|Y(x;x_0)P_4Y^{-1}(s;x_0)\| \\
+(\kappa x)^{1+\beta}\|Y(x;x_0)\Pi Y^{-1}(s;x_0)\|\bigr) \leq C_4(x_0)
x^{1+\alpha} e^{-\gamma s}.\end{gather*}

 5) Let $x_0\le x<s.$ In this case we also have $G(x,s) = -
Y(x;x_0)\cdot (P_4+\Pi)Y^{-1}(s;x_0)$. Hence,
\begin{gather*}
J(s;x;x_0)\leq \|Y(x;x_0)P_1\|\bigl(\|Y(x;x_0)P_4Y^{-1}(s;x_0)\|\\ +
\|Y(x;x_0)\Pi Y^{-1}(s;x_0)\|\bigr)  \leq
2C_*^2e^{-\gamma(x-x_0)}e^{-\gamma(s-x)}= C_5(x_0)e^{-\gamma s}
\end{gather*}
where $C_5(x_0):=2C_*^2e^{\gamma x_0}$.

 6) Finally, let $x_0\le s \le x.$ Now $G(x,s) =
Y(x;x_0)(P_1+P_2+P_3)Y^{-1}(s;x_0),$ and there exists a constant
$C_6(x_0)>0$ such that
\begin{gather*}
J(s;x;x_0) \leq
\|Y(x;x_0)P_1\|\bigl(\|Y(x;x_0)(P_1+P_2+P_3)Y^{-1}(s;x_0)\|\\ +
\|N(x;x_0)\Pi Y^{-1}(s;x_0)\|\bigr) \leq
C_6(x_0)\left(e^{-\gamma(2x-s)}+e^{-\gamma(x+s)}\right).
\end{gather*}

The above arguments  prove  the following theorem.
\begin{theorem}\label{th:exsolhombvp}
There exists a solution to the boundary value problem
(\ref{eq:nonhom1})--(\ref{eq:hombc}) iff the function $g(x)$ is
orthogonal (in terms of the scalar product $\langle \cdot,\cdot \rangle_{L_2}:=
\int_{0}^{\infty}\langle \cdot,\cdot \rangle \,dx$) to all solutions to the adjoint system
(\ref{eq:conjsys}) which belong to $L_1([0,\infty) \mapsto \mathbb R^n)$. If the
orthogonality condition holds true, then the problem
(\ref{eq:nonhom1})--(\ref{eq:hombc}) has the family of solutions
which can be represented as the sum of two mutually orthogonal
compo\-nents
\begin{gather*}
  y=Y(x;x_0)v +\int_{0}^{\infty}\mathfrak G(x,s)g(s)\,ds
\end{gather*}
where $v \in \mathbb L_1$ is an arbitrary vector and $\mathfrak G(x,s)$ is the
generalized Green function defined by (\ref{eq:gengrfunc}).
\end{theorem}

\section{The main theorem}\label{sec:MainRes}
Let us turn back to the main problem of finding solutions to the
system (\ref{eq:l-s-sys}) which possess the properties
(\ref{eq:l-s-bvp}). It is clear that a continuously differentiable on
$[0,\infty)$ solution $y(x)$ to the system (\ref{eq:l-s-sys}), provided
that it exists, must satisfy the equality $Ay(+0)+a=0$. Thus we
require the following condition to hold true:

\textbf{C}: $a\in \mathrm{im}\,A$.

 The orthogonal decomposition  $\mathbb R^n=\mathrm{im}\,A^*\oplus \ker A$
together with the condition \textbf{C} imply existence of the unique
$\eta \in \mathrm{im}\,A^*$ for which $A\eta +a=0$.

Hence, it is naturally to formulate \emph{main boundary value
problem} in the following way: \textsl{ find all $\zeta \in \ker A$ for
which the boundary value problem for the system (\ref{eq:l-s-sys})
with the boundary conditions
\begin{gather*}
  y(+0)=\eta +\zeta,\quad y(\infty)=0
\end{gather*}
is solvable in the class $C^1\big([0,\infty) \mapsto \mathbb
R^n\big)$ and construct an integral represent\-ation of
corresponding solutions.} This problem is resolved by the following
theorem.
\begin{theorem}\label{th:mainthirs}
Let the system (\ref{eq:l-s-sys}) satisfy the conditions
$\mathbf{A}$ -- $\mathbf{C}$ and $f(x)\to 0$ when $x\to +\infty $.
Then the main boundary value problem is solvable iff
\begin{equation}\label{eq:ortcondP6}
  \int_{0}^{\infty}P_6Y^{-1}(x;x_0) \big[f(x)+B(x)\eta \big]\,dx=0.
\end{equation}

Provided that (\ref{eq:ortcondP6}) holds true,  the main boundary
value problem has the family of solutions defined by the formulae
\begin{gather}\label{eq:irs}
y=Y(x;x_0)(v_1+v_2)+\eta +\int_{0}^{\infty}G(x,s)\big(f(s)+B(s)\eta \big)\,ds,\\
  \zeta =(E+U(x_0))^{-1}v_2+(E+\Theta(x_0))^{-1}w\label{eq:zeta}
\end{gather}
where $v_1\in \mathbb L_1,\;v_2\in \mathbb L_2$ are arbitrary
vectors and
\begin{equation}
  \label{eq:ortcondP5}
w:=-\int_{0}^{\infty}P_5Y^{-1}(x;x_0) \big[f(x)+B(x)\eta\big]\,dx.
\end{equation}

There exist positive constants $K_1(\alpha,\gamma,x_0),\,K_2(\alpha,\gamma,x_0)$ such
that
\begin{gather*}
  \int_{0}^{\infty}\bigl\|G(x,s)\big(f(s)+B(s)\eta \big)\bigr\|\,ds\le
K_1(\alpha,\gamma,x_0)\sup_{x\in [0,\infty)}\|f(s)+B(s)\eta \|,\\ \|w\|\le
K_2(\alpha,\gamma,x_0)\sup_{x\in [0,\infty)}\|f(s)+B(s)\eta \|.
\end{gather*}

\end{theorem}

\begin{proof}
We seek the solution to the problem
(\ref{eq:l-s-sys})--(\ref{eq:l-s-bvp}) in the form
\begin{gather}\label{eq:y-to-z}
  y=\eta +\varphi(x)+Y(x;x_0)v+e^{-\kappa x}Y(x;x_0)w +y_0(x)
\end{gather}
where $\kappa>\gamma $ is an arbitrary number, $v\in \mathbb L_1\oplus\mathbb L_2$ is an
arbitrary constant vector, $w\in \mathbb L_5$ is a constant vector which is
to be determined, $\varphi(\cdot)\in C^1([0,\infty) \mapsto \mathbb R^n)$ is an arbitrary
function with the properties $$\varphi(0)=0,\quad \lim_{x\to
\infty}\varphi(x)=-\eta,\quad \lim_{x\to \infty}\varphi'(x)=0,$$ and $y_0(x)$ is a
solution to the problem (\ref{eq:nonhom1})--(\ref{eq:hombc}) with
$$g(x):=f(x)+B(x)\eta -\mathfrak L\varphi(x)+\kappa e^{-\kappa x}Y(x;x_0)w\quad
  \text{when}\;
  x>0.$$
Observe that there exists $\lim_{x\to+0}g(x)$. In virtue of the
Theorem \ref{th:exsolhombvp}, existence of the solution $y_0(x)$ is
guaranteed by the orthogo\-nality conditions which can be given in
the form
\begin{gather*}
\int_{0}^{\infty}\langle[Y^{-1}(s;x_0)]^*P_5b,f(s)+B(s)\eta
-\mathfrak L\varphi(s)\rangle \,ds+w=0,\\
\int_{0}^{\infty}\langle[Y^{-1}(s;x_0)]^*P_6b,f(s)+B(s)\eta
-\mathfrak L\varphi(s)\rangle \,ds=0,\quad \forall b\in \mathbb R^n.
\end{gather*}
Since $\langle[Y^{-1}(s;x_0)]^*P_jb,\varphi(s)\rangle \big|_{s=0}^\infty =0,\;j=5,6$,
and $\mathfrak L^*[Y^{-1}(s;x_0)]^*=0$, these conditions are
equivalent to (\ref{eq:ortcondP5}), (\ref{eq:ortcondP6}). The
orthogonality conditions also imply that
\begin{gather*}
  y_0(x):=\int_{0}^{\infty}\mathfrak G(x,s)g(s)\,ds \\=\int_{0}^{\infty}G(x,s)g(s)\,ds+
  Y(x;x_0)P_1\int_{0}^{\infty}M(s;x_0)g(s)\,ds.
\end{gather*}
The second addend is inessential owing to the presence of an
arbitrary vector $v\in \mathbb L_1\oplus\mathbb L_2$ in the formula
(\ref{eq:y-to-z}).

Next, it is not hard to show that $$\int_{0}^{\infty}G(x,s)\mathfrak
L\varphi(s)\,ds= \varphi(x)-Y(x;x_0)P_1\varphi(x_0)$$ and
\begin{gather*}
\int_{0}^{\infty}G(x,s)\kappa e^{-\kappa s}Y(s;x_0)w\,ds=-e^{-\kappa
x}Y(x;x_0)w.
\end{gather*}
Taking into account these equalities, one can rewrite the formula
(\ref{eq:y-to-z}) in the form (\ref{eq:irs}). Finally, in view of
(\ref{fund}), (\ref{eq:V0}), (\ref{eq:y-to-z}), and the equality
$y_0(+0)=0$ we obtain (\ref{eq:zeta}).

Now, observe that from the definition of $\mathbb L_5$ it follows
that the constant $C_7(x_0):=\max_{x\in [0,x_0]}\|Y(x;x_0)P_5\|$ is
well defined.  Let $\bar{g}(s)$ $:=f(s)+B(s)\eta $. Making use of
\eqref{eq:ortcondP6} and estimates similar to those which were
obtained for the function $J(s;x;x_0)$ in previous section, we have:

1) if $0\le x\le x_0$, then
\begin{gather*}
\int_{0}^{\infty}\bigl\|G(x,s)\bar{g}(s)\bigr\|\,ds \le \sup_{s\in
[0,\infty)}\|\bar{g}(s)\|
\biggl(\int_{0}^{x}\|Y(x;x_0)(P_2+P_3)Y^{-1}(s;x_0)\|\,ds \\ +
\int_{x}^{x_0}\|Y(x;x_0)(P_1+P_4+P_5)Y^{-1}(s;x_0)\|\,ds\\ +
\int_{x_0}^{\infty}\|Y(x;x_0)(P_4+P_5)Y^{-1}(s;x_0)\|\,ds\biggr) \\ \le
\sup_{s\in [0,\infty)}\|\bar{g}(s)\|
\biggl(C_0\int_{0}^{x}\!\!(x/s)^{1-\alpha}\,ds+
C_0\int_{x}^{x_0}\!\!\bigl((x/s)^{1+\alpha}+C_7(x_0)(x_0/s)^{1-\alpha}\bigr)\,ds
\\ +C_*\int_{x_0}^{\infty}\!\!\bigl(C_0\cdot (x/x_0)^{1+\alpha}+C_7(x_0)\bigr)
e^{-\gamma(s-x_0)}\,ds\biggr)\le K_1(\alpha,\gamma,x_0)\sup_{s\in
[0,\infty)}\|\bar{g}(s)\|;
\end{gather*}

2) if $0< x_0< x$, then
\begin{gather*}
\int_{0}^{\infty}\bigl\|G(x,s)\bar{g}(s)\bigr\|\,ds \le \sup_{s\in
[0,\infty)}\|\bar{g}(s)\|
\biggl(\int_{0}^{x_0}\!\!\|Y(x;x_0)(P_2+P_3)Y^{-1}(s;x_0)\|\,ds\\ +
\int_{x_0}^{x}\|Y(x;x_0)(P_1+P_2+P_3)Y^{-1}(s;x_0)\|\,ds\\ +
\int_{x}^{\infty}\|Y(x;x_0)(P_4+P_5)Y^{-1}(s;x_0)\|\,ds\biggr) \\ \le
\sup_{s\in [0,\infty)}\|\bar{g}(s)\|
\biggl(C_0C_*\int_{0}^{x_0}(x_0/s)^{1-\alpha}e^{-\gamma(x-x_0)}\,ds+
C_*\int_{x_0}^{x}e^{-\gamma(x-s)}\,ds + \\ +
C_*\int_{x}^{\infty}e^{-\gamma(s-x)}\,ds\biggr)\le K_1(\alpha,\gamma,x_0)\sup_{s\in
[0,\infty)}\|\bar{g}(s)\|.
\end{gather*}

Finally, the inequality for $\|w\|$ is easily obtained with help of
the estimates from the proof of Proposition \ref{pr:intsol}.
\end{proof}

\textbf{Conclusions.} The results obtained can be interpreted in
terms of linear equations in Banach spaces in the following way. Let
$Y$ be the Banach space of continuous mappings $y(\cdot):[0,\infty)
\mapsto \mathbb R^{n}$ such that $\lim_{x\to +\infty}y(x)=0$, and
$X\subset Y$ be the Banach space of mappings satisfying $y(0)=0$
(these spaces are endowed with usual supremum norm). Consider the
closed linear operator $\mathcal{L}:X \mapsto Y$ defined on the
dense domain $D(\mathcal{L})=\{y(\cdot)\in X\cap C^1([0,\infty)
\mapsto \mathbb R^{n}):\lim_{x\to +\infty}y'(x)=0\}$ by
$\mathcal{L}y(x):=y'(x)-Ay(x)/x-B(x)y(x)$. From Proposition
\ref{pr:famrbs} it follows that the range $R(\mathcal{L})$ is a
closed subspace of Banach space $Y$. Hence, the operator
$\mathcal{L}$ is normally solvable, moreover, it is both $n$-normal
with $n(\mathcal{L})=\dim\ker\mathcal{L}=\dim \mathbb L_1$ and
$d$-normal with
$d(\mathcal{L})=\mathrm{codim}\,R(\mathcal{L})=\dim(\mathbb
L_5+\mathbb L_6)$. This means that we have established conditions
under which the operator $\mathcal{L}$ is a Noetherian one with
index $n(\mathcal{L})-d(\mathcal{L})$.

\end{document}